\documentclass[12pt]{amsart}

\usepackage{amsmath}
\usepackage{amscd}
\usepackage{eucal}
\usepackage{amssymb}
\usepackage{epic}
\usepackage{graphics}
\usepackage{color}
\usepackage{mathrsfs}
\usepackage{multirow}
\usepackage{subfigure}
\usepackage{graphicx}
\usepackage{mathtools}

\numberwithin{equation}{section}
\numberwithin{table}{section}

\theoremstyle{plain}
\newtheorem{theorem}{Theorem}[section]
\newtheorem*{thm:A}{Theorem \ref{thm:A}}
\newtheorem*{thm:B}{Theorem \ref{thm:B}}
\newtheorem{lemma}{Lemma}[section]

\theoremstyle{remark}

\newcommand{\Zp}{{\mathbb{Z}_{p}}}
\newcommand{\Qp}{{\mathbb{Q}_{p}}}

\newcommand{\ZZ}{{\mathbb{Z}}}

\newcommand{\allone}{{\mathbf{1}}}
\newcommand{\Z}{{\mathbf{Z}}}

\DeclareMathOperator{\im}{Im}

\DeclareMathOperator{\rank}{rank}
\DeclareMathOperator{\Ker}{Ker}
\DeclareMathOperator{\End}{End}
\DeclareMathOperator{\Sym}{Sym}

\title[Critical group of $KG(n,2)$]{The critical group of the Kneser graph on $2$-subsets of an $n$-element set.}
\author[Ducey]{Joshua E. Ducey}
\author[Hill]{Ian Hill}
\author[Sin]{Peter Sin}

\address{Dept.\ of Mathematics and Statistics, James Madison University, Harrisonburg, VA 22807, USA}
\email{duceyje@jmu.edu}
\email{hillim@dukes.jmu.edu}
\address{Dept.\ of Mathematics, University of Florida, P.O. Box 118105, Gainesville, FL 32611, USA}
\email{sin@ufl.edu}
\thanks{This work was partially supported by a grant from the Simons Foundation (\#204181 to Peter Sin). }


\begin{document}
\begin{abstract}
In this paper we compute the critical group of the Kneser graph $KG(n,2)$.  This is equivalent to computing the Smith normal form of a Laplacian matrix of this graph.

\end{abstract}
\maketitle
\section{Introduction}
Given a finite graph $\Gamma$, it is of interest to compute linear algebraic invariants that arise from certain matrices attached to $\Gamma$.  Two of the most common matrices are the adjacency matrix $A$ and the Laplacian matrix $L$.  The spectrum, for example, of the matrices does not depend on the ordering of the vertices.  Another graph invariant which we will consider here is the cokernel of the matrix viewed as a map of free abelian groups.  A cyclic decomposition of the cokernel is described by the Smith normal form of the matrix.  

When the matrix is the adjacency $A$, the cokernel is called the \textit{Smith group} $S(\Gamma)$.  The torsion subgroup of the cokernel of $L$ is known as the \textit{critical group} $K(\Gamma)$.  The critical group is especially interesting; for a connected graph, the order of $K(\Gamma)$ counts the number of spanning trees of the graph.  The critical group is also sometimes called the \textit{sandpile group}, or \textit{Jacobian group} \cite{lorenzini}.

Let $\Gamma = KG(n,2)$, the Kneser graph with vertex set equal to the $2$-element subsets of an $n$-element set.  A pair of $2$-element subsets (we will abbreviate to ``$2$-subsets'') are adjacent when they are disjoint.  The complement $\Gamma^{\prime}$ of this graph is called the \textit{triangular graph} $T(n)$.  The Smith group $S(\Gamma^{\prime})$ was computed in \cite[SNF3]{bve} by using integral row and column operations on the adjacency $A^{\prime}$ to produce a diagonal form.  The critical group $K(\Gamma^{\prime})$ was described in \cite[Cor. 9.1]{crit-line}, where the authors study critical groups of line graphs (note that $\Gamma^{\prime}$ is the line graph of the complete graph on $n$ vertices).   In \cite{wilson} an integral basis was found for the inclusion matrices of $r$-subsets vs.\ $s$-subsets that puts them in diagonal form.  Observe that a pair of subsets are disjoint precisely when one is included in the complement of the other, thus the adjacency $A$ of $\Gamma$ is integrally equivalent to the $2$-subsets vs.\ $(n-2)$-subsets inclusion matrix.  In general, the Smith group of any Kneser graph can be deduced from the work \cite{wilson}. 

In this paper we will compute the critical group $K(\Gamma)$ of the Kneser graph on $2$-subsets of an $n$-element set.  It is worth drawing attention to the quite different approaches taken in calculating the three families of groups above.  In this work we will demonstrate yet another technique for calculating Smith/critical groups by applying some representation theory of the symmetric groups.

The paper is organized as follows.  In Section \ref{sec:prelims} we give formal definitions and describe a few important results and notions that will be used repeatedly.  In Section \ref{sec:elem} we compute the elementary divisor decomposition of the critical group $K(\Gamma)$ by considering each prime $p$ that divides the order of the group.  In Section \ref{sec:inv} we put this information together and obtain a concise description of the Smith normal form of $L$.  

\section{Preliminaries} \label{sec:prelims}

\subsection{Smith normal form}

Let $\Gamma$ be a finite graph.  Under some ordering of the vertex set $V(\Gamma)$, define the adjacency matrix $A = (a_{i,j})$ by
\[
a_{i,j} = \begin{cases}
1, & \mbox{if vertex $i$ is adjacent to vertex $j$}\\
 0, & \mbox{otherwise.}
\end{cases}
\]

Let $D$ denote the diagonal matrix of the same size at $A$, with the $(i,i)$ entry equal to the degree of vertex $i$.  Then the Laplacian matrix of $\Gamma$ is 
\[
L = D - A.
\]
We will view $L$ as defining a map of free abelian groups
\[
L \colon \ZZ^{V(\Gamma)} \to \ZZ^{V(\Gamma)}
\]
and we will use the same symbol for both the matrix and the map.

The critical group of $\Gamma$, denoted $K(\Gamma)$, is the torsion subgroup of $\ZZ^{V(\Gamma)} / \im(L)$.  For $\Gamma$ a connected graph, we have 
\[
\ZZ^{V(\Gamma)} / \im(L) \cong K(\Gamma) \oplus \ZZ.
\]
The critical group is also described by the Smith normal form of $L$, which we now review.

Let $M$ be an $m \times n$ integer matrix.  It is well known \cite[Ch. 12]{df} that there exist unimodular (i.e., unit determinant) matrices $U$ and $V$ so that $UMV = S$, where $S = (s_{i,j})$ is an $m \times n$ matrix satisfying:
\begin{enumerate}
\item $s_{i, i}$ divides $s_{i+1,i+1}$ for $i < \rank(M)$,
\item $s_{i,i} = 0$ for $i > \rank(M)$,
\item $s_{i,j} = 0$ for $i \neq j$.
\end{enumerate}
The entries of the matrix $S$ are unique up to multiplication by units, and $S$ is called the Smith normal form of $M$.  Viewing $M$ as defining a map:
\[
M \colon \ZZ^{n} \to \ZZ^{m},
\]
we have that 
\[
\ZZ^{m} / \im(M) \cong \ZZ / s_{1,1}\ZZ \oplus \ZZ / s_{2,2}\ZZ \oplus \cdots \oplus \ZZ / s_{r, r}\ZZ \oplus \ZZ^{m-r},
\]
where $r = \rank(M)$.  Thus the Smith normal form provides the \textit{invariant factor decomposition} of the finitely generated abelian group $\ZZ^{m} / \im(M)$.  Of course, any other diagonal form of $M$ achievable by unimodular matrices serves equally well to identify the group as a direct product of cyclic groups.  

A finite abelian group also has a decomposition into cyclic groups of prime-power orders.  This \textit{elementary divisor decomposition} is particularly accessible to us by using ideas which we now describe.  It is natural to work to identify the $p$-elementary divisors for each prime $p$.

When considering the $p$-elementary divisors of $M$, it is convenient to work over the $p$-adic integers, $\Zp$.  Any integer matrix can be viewed as a matrix with entries coming from $\Zp$.  The above theory of Smith normal form still holds, as $\Zp$ is a PID.  As each prime integer different from $p$ becomes a unit in $\Zp$, the entries $s_{i,i}$ in the Smith normal form of $M$ may be replaced with the largest power of $p$ that divides $s_{i,i}$.  In particular, the $p$-elementary divisor multiplicities of $M$ are the same whether we view $M$ as a matrix over $\ZZ$ or $\Zp$.  

Fix a prime $p$ and view $M$ as a matrix over $\Zp$.  We let $e_i$ denote the number of times $p^{i}$ occurs as an entry in the Smith normal form of $M$; that is, the multiplicity of $p^{i}$ as an elementary divisor of $M$.  (In the next section we will compute these numbers for the Laplacian matrix $L$ of the Kneser graph on $2$-subsets.)  We relate these numbers $e_{i}$ to some $\Zp$-submodules of the domain of $M \colon \Zp^{n} \to \Zp^{m}$ that we now define.

Let $M_{i} = M_{i}(M) = \{x \in \Zp^{n} \, \vert \, p^{i} \mbox{ divides } Mx\}$.  Then we have
\[
\Zp^{n} = M_{0} \supseteq M_{1} \supseteq M_{2} \supseteq \cdots \supseteq \Ker(M).
\]
For a $\Zp$-submodule $N \subseteq \Zp^{n}$ we set $\overline{N} = (N + p\Zp)/p\Zp$, and set $F = \ZZ/p\ZZ$.  Then by reduction modulo $p$ we have an induced map $\overline{M} \colon F^{n} \to F^{m}$, as well as an induced chain of $F$-vector spaces
\[
F^{n} = \overline{M_{0}} \supseteq \overline{M_{1}} \supseteq \cdots \supseteq \overline{\Ker(M)}.
\]
It is well-known (for example, \cite[Lemma 3.1]{moore}) that 
\begin{equation} \label{eqn-mdim}
\dim_{F} \overline{M_{i}} = \dim_{F} \overline{\Ker(M)} + \sum_{j \geq i} e_{j}.
\end{equation}

\subsection{Graph parameters}

The Kneser graph $KG(n,2)$ is the graph with vertex set equal to the $2$-subsets of an $n$-element set.  A pair of $2$-subsets are adjacent if and only if they are disjoint.  

One quickly checks that $KG(2,2)$ is the graph on one vertex, $KG(3,2)$ is the graph on three vertices with no edges, and $KG(4,2)$ is the disjoint union of three copies of the path on two vertices.  In all these cases the critical group is trivial.

Thus we will assume $n\geq 5$, and we will denote this graph by $\Gamma$ for the remainder of the paper.

This graph is strongly regular \cite[9.1.3]{bh} with parameters
\begin{align*}
v &= \binom{n}{2}\\
k &= \binom{n-2}{2}\\
\lambda &= \binom{n-4}{2}\\
\mu &= \binom{n-3}{2}.
\end{align*}
Here $v$ is the number of vertices of $\Gamma$ and the other parameters can be defined by the equation
\begin{equation} \label{eqn:srgA}
A^{2} = kI + \lambda A + \mu (J - A - I),
\end{equation}
where $A$ is the adjacency matrix, $I$ is the identity matrix, and $J$ is the matrix of all-ones.

Let $L$ denote the Laplacian matrix for $\Gamma$.  The eigenvalues of $L$ are:
\[
r \coloneqq \frac{n(n-3)}{2}
\]
occurring with multiplicity $f \coloneqq \binom{n}{1}-\binom{n}{0} = n-1$,
\[
s \coloneqq \frac{(n-4)(n-1)}{2}
\]
occurring with multiplicity $g \coloneqq \binom{n}{2}-\binom{n}{1} = \frac{n(n-3)}{2}$, and $0$ occurring with multiplicity $1$.

As a consequence of equation \ref{eqn:srgA} we have
\begin{equation} \label{eqn:srgL}
(L - rI)(L - sI) = \mu J.
\end{equation}

By the Matrix-Tree Theorem the critical group of $\Gamma$ has order
\begin{equation} \label{crit-order}
|K(\Gamma)| = \frac{\left[ \frac{n(n-3)}{2}\right]^{f} \left[ \frac{(n-4)(n-1)}{2}\right]^{g}}{\frac{n(n-1)}{2}} = \frac{n^{f-1}(n-1)^{g-1}(n-3)^{f}(n-4)^{g}}{2^{f+g-1}}.
\end{equation}

\subsection{Module lemmas}

Let $\Qp$ denote the field of fractions of $\Zp$. 

\begin{lemma} \label{lem:eigenspace}
Let $X$ be a free $\Zp$-module of finite rank and let $\phi \in \End_{\Zp}(X)$.  Let $\tilde{\phi}$ denote the induced endomorphism of $\Qp \otimes_{\Zp} X$.  Suppose that $\tilde{\phi}$ has an integer eigenvalue $u$, the $u$-eigenspace $V_{u}$ has dimension $b$, and $p^{a} \mid u$.  Then $\dim_{F} \overline{M}_{a} \geq b$.
\end{lemma}
\begin{proof}
The submodule $X \cap V_{u}$ of $X$ is pure \cite[pg. 84]{cr}, of rank $b$, and lies in $M_{a}$.  Thus $b = \dim_{F} \overline{X \cap V_{u}} \leq \dim_{F} \overline{M}_{a}$.  \end{proof}

The symmetric group $\Sym(n)$ has a natural action on the $2$-subsets of an $n$-element set.  This action makes $\Zp^{V(\Gamma)}$ into a $\Zp \Sym(n)$-permutation module, and the Laplacian a module endomorphism.  The subgroups $M_{i}$ become submodules.  Analogous statements hold when we extend scalars to $\Qp$ or reduce modulo $p$.  Thus we can make use of certain results from the representation theory of $\Sym(n)$.  We summarize what we will need below.

Let $\Gamma_{i}$ denote the graph with vertices the $i$-subsets of an $n$-element set, adjacent when disjoint.  It was shown in \cite[Ch. 17]{james} that the  $F \Sym(n)$-permutation module $F^{V(\Gamma_{i})}$ has a descending filtration with subquotients isomorphic to the \textit{Specht modules} $S^{j}$, for $0 \leq j \leq i$.  When $n>2i$, each Specht module $S^{j}$ has a unique maximal submodule with simple quotient denoted $D^{j}$.  Furthermore, the only other possible composition factors of $S^{j}$ are $D^{k}$, $k < j$, each occurring with multiplicity zero or one.  These composition factor multiplicities are given in the lemma below.  While the majority of our arguments will use only elementary linear algebra, this lemma will be useful to us in several places.  In \cite{james} the reader will find the notation $M^{(n-i, i)}$, $S^{(n-i,i)}$, $D^{(n-i,i)}$ instead of $F^{V(\Gamma_{i})}$, $S^{i}$, $D^{i}$, respectively.

\begin{lemma}[\cite{james}, Theorem 24.15] \label{lem:james}
Let $n \geq 5$ and $n>2i$.  Let $\left [ S^{i} : D^{j} \right ]$ denote the multiplicity of $D^{j}$ as a composition factor of $S^{i}$.  
\begin{enumerate}
\item $\left [ S^{1} : D^{0} \right ] = 1$ if $n \equiv 0 \pmod{p}$, and is $0$ otherwise.
\item $\left [ S^{2} : D^{1} \right ] = 1$ if $n \equiv 2 \pmod{p}$, and is $0$ otherwise.
\item Let $p>2$.  Then $\left [ S^{2} : D^{0} \right ] = 1$ if $n \equiv 1 \pmod{p}$, and is $0$ otherwise.
\item Let $p = 2$.  Then $\left [ S^{2} : D^{0} \right ] = 1$ if $n \equiv 1 \mbox{ or } 2 \pmod{4}$, and is $0$ otherwise.
\end{enumerate}
\end{lemma}

We make one more observation.  Viewing the Laplacian $L$ as a map over $\Qp$, the eigenspaces are $\Qp \Sym(n)$-modules, and it is not hard to see that the $r$-eigenspace and $s$-eigenspace must be the irreducible Specht modules over $\Qp$; we denote them $\tilde{S^{1}}$ and $\tilde{S^{2}}$, respectively.  By a general principle \cite[Prop. 16.16]{cr} in modular representation theory, the $F \Sym(n)$-modules $\overline{\Zp^{V(\Gamma)} \cap \tilde{S^{i}}}$ and $S^{i}$ must share the same multiset of composition factors since they are both reductions modulo $p$ of $\Zp$-forms of $\tilde{S^{i}}$.  In particular $\dim_{F} S^{1} = f$ and $\dim_{F} S^{2} = g$, and we are able to use Lemma \ref{lem:james} to work out the dimensions of the simple modules $D^{i}$ in each case.


\section{The elementary divisors of $L$} \label{sec:elem}
Now we will compute the $p$-elementary divisors of $L$, for each prime $p$ that divides the order of the critical group of $\Gamma$.  If $p^{i} \mid m$, but $p^{i+1} \nmid m$, we will write $v_{p}(m) = i$.  We also remind the reader of our notation for the eigenvalues $r \coloneqq n(n-3)/2$ and $s \coloneqq (n-4)(n-1)/2$ of $L$, eigenspace dimensions $f \coloneqq n-1$ and $g \coloneqq n(n-3)/2$, and that $e_{i}$ denotes the multiplicity of $p^{i}$ as an elementary divisor of $L$. 

Let $p \mid |K(\Gamma)|$.
\vspace{.5cm}

\underline{Case 1}: $p > 3$.  Observe from equation \ref{crit-order} that if $p \mid |K(\Gamma)|$ and $p \neq 2, 3$ then $p$ must divide exactly one of $n, n-1, n-3, n-4$.

\begin{enumerate}
\item[a) ] $v_{p}(n) = a$.   We will need to establish that $\dim_{F} \overline{M}_{a} \geq f$.  This follows immediately from Lemma \ref{lem:eigenspace} since the eigenvalue $r$ satisfies $v_{p}(r) = a$.  Now by equations \ref{crit-order} and \ref{eqn-mdim} we observe:
\begin{align*}
a(f-1) &= v_p(|K(\Gamma)|) \\
&= \sum_{i\geq 0} ie_i \\
&\geq \sum_{i \geq a} ie_{i} \\
&\geq a\sum_{i \geq a} e_i \\
&= a(\dim_{F}\overline{M}_a -1) \\
&= a\dim_{F}\overline{M}_a -a \\
&\geq af - a
\end{align*}
and so we must in fact have equality throughout.  The fact that the top-most inequality in the chain above is an equality implies that $e_{i} = 0$ for $1 \leq i < a$.  Considering the next inequality above, equality now forces that $e_{i} = 0$ for $i > a$.  In the last inequality in the chain above, equality forces that $\dim_{F} \overline{M}_{a} = f$, and so $e_{a} = f-1$ by equation \ref{eqn-mdim}.  Finally, since $\sum_{i \geq 0} e_{i} = f+g$, we have shown that
\begin{align*}
e_{a} &= f-1,\\
e_0 &= g+1, \\
e_i &= 0, \mbox{ otherwise.} 
\end{align*}

\end{enumerate}

We pause the development to indicate to the reader how the proofs of the remaining cases will proceed.  In each case we will first establish lower bounds for the dimensions of various modules $\overline{M}_{i}$, and then a similar chain of inequalities as the one above will force nearly all of the multiplicities $e_{i}$ to equal zero.  In fact, this argument has been abstracted to a lemma in \cite{grassmann}, which we now state as it applies to our situation.
\begin{lemma}[\cite{grassmann}, Lemma 3.1] \label{lem:grassmann}
Set $d = v_{p}(|K(\Gamma)|)$.
Suppose  that we have an increasing sequence of indices 
$0<a_1<a_2<\cdots<a_h$ and a corresponding sequence of lower bounds
$b_1>b_2>\cdots>b_h$ satisfying the following conditions. 
\begin{enumerate}
\item$\dim_F\overline M_{a_j}\geq b_j$ for $j=1$,\ldots, $h$. 
\item $\sum_{j=1}^{h}(b_j-b_{j+1})a_j=d$, where we set $b_{h+1}=\dim_F\overline{\Ker(L)} = 1$. 
\end{enumerate}
Then the following hold.
\begin{enumerate}
\item $e_{a_j}=b_j-b_{j+1}$ for $j=1$,\ldots, $h$. 
\item $e_0=\dim_F\overline{M}_{0}-b_1 = f + g + 1 - b_{1}$.
\item $e_i=0$ for $i\notin\{0, a_1,\ldots, a_h\}$.
\end{enumerate}
\end{lemma}

By applying Lemma \ref{lem:grassmann} with $h = 1$, $a_{1} = a$, $b_{1} = f$, $b_{2} = 1$ and $d = a(f-1)$, we recover 
\begin{align*}
e_{a} &= f-1,\\
e_0 &= g+1, \\
e_i &= 0, \mbox{ otherwise.} 
\end{align*}

\begin{enumerate}
\item[b) ] $v_{p}(n-1) = a$.  In this case we will need the bound $\dim_{F} \overline{M}_{a} \geq g$.  This follows from Lemma \ref{lem:eigenspace} since the eigenvalue $s$ satisfies $v_{p}(s) = g$.  Now we apply Lemma \ref{lem:grassmann} with $h = 1$, $a_{1} = a$, $b_{1} = g$, $b_{2} = 1$, and $d = a(g-1)$ to conclude
\begin{align*}
e_{a} &= g-1,\\
e_{0} &= f+1, \\
e_{i} &= 0, \mbox{ otherwise.}
\end{align*}

\item[c) ] $v_{p}(n-3) = a$.  In this case we will need the bound $\dim_{F} \overline{M}_{a} \geq f+1$.  We immediately get that $\dim_{F} \overline{M}_{a} \geq f$ by Lemma \ref{lem:eigenspace} applied to the eigenvalue $r$.  By Lemma \ref{lem:james} 
we see that the trivial module $D^{0}$ is not a composition factor of $S^{1}$ since $p \nmid n$.  But the trivial module spanned by the all-one vector (that is, $\overline{\Ker(L)}$) is contained in $\overline{M}_{a}$, and so we must have $\dim_{F} \overline{M}_{a} \geq f+1$.

Now we apply Lemma \ref{lem:grassmann} with $h = 1$, $a_{1} = a$, $b_{1} = f+1$, $b_{2} = 1$, and $d = af$ to conclude that 
\begin{align*}
e_{a} &= f,\\
e_{0} &= g,\\
e_{i} &= 0, \mbox{ otherwise.}
\end{align*}

\item[d) ] $v_{p}(n-4) = a$.  In this case we need the bound $\dim_{F} \overline{M}_{a} \geq g+1$.  By Lemma \ref{lem:eigenspace} applied to the eigenvalue $s$, we have that $\dim_{F}\overline{M}_{a} \geq g$.  
By Lemma \ref{lem:james} we see that the trivial module is not a composition factor of $S^{2}$ since $p \nmid n-1$.  However, since the trivial module spanned by the all-one vector lies in $\overline{M}_{a}$, we must have $\dim_{F} \overline{M}_{a} \geq g+1$.  Applying Lemma \ref{lem:grassmann} with $h=1$, $a_{1}=a$, $b_{1} = g+1$, $b_{2} = 1$, and $d = ag$ we conclude that 
\begin{align*}
e_{a} &= g,\\
e_{0} &= f,\\
e_{i} &= 0, \mbox{ otherwise.}
\end{align*}

\end{enumerate}

\underline{Case 2}: $p = 3$.  We have six distinct situations.

\begin{enumerate}
\item[a) ]  $v_{3}(n-1) = a$, with $a>1$.  Then we must have $v_{3}(n-4) = 1$.  We will need the bounds $\overline{M}_{1} \geq g+1$ and $\overline{M}_{a+1} \geq g$.  The second bound follows immediately from Lemma \ref{lem:eigenspace} applied to the eigenvalue $s$.  To establish the first bound, consider equation \ref{eqn:srgL}.  Modulo $3$, this equation reads
\begin{equation} \label{eqn:Lbar}
(\overline{L} - \overline{r}I)\overline{L} = 0.
\end{equation}
Since $r \neq 0 \pmod 3$ and $s = 0 \pmod 3$, the algebraic multiplicity of $0$ as an eigenvalue of $\overline{L}$ is $g+1$.  Equation \ref{eqn:Lbar} shows that $\overline{L}$ is diagonalizable.  Therefore, $\overline{M}_{1} = \Ker(\overline{L})$ has dimension $g+1$.  

By Lemma \ref{lem:grassmann} with $h=2$, $a_{1}=1$, $a_{2} = a+1$, $b_{1} = g+1$, $b_{2} = g$, $b_{3} = 1$, and $d = a(g-1)+g$, we conclude that
\begin{align*}
e_{1} &= 1,\\
e_{a+1} &= g-1,\\
e_{0} &= f,\\
e_{i} &= 0, \mbox{ otherwise.}
\end{align*}

\item[b) ]  $v_{3}(n-4) = a$, with $a>1$.  Then we must have $v_{3}(n-1) = 1$.  In this case we need to establish the bounds $\dim_{F}\overline{M}_{a} \geq g+1$ and $\dim_{F}\overline{M}_{a+1} \geq g$.  The second bound immediately follows from Lemma \ref{lem:eigenspace} applied to $s$.  To get the first bound, first consider equation \ref{eqn:srgL} which we rewrite as
\[
L(L - rI) = s(L-rI) + \mu J.
\]
Since both $s$ and $\mu$ are divisible by $3^{a}$, this shows that $\im(L - rI) \subseteq M_{a}$.  Therefore $\im(\overline{L} - \overline{r}I) \subseteq \overline{M}_{a}$.  Just as in the previous case, $\overline{L}$ is diagonalizable with $\dim_{F}\Ker(\overline{L} - \overline{r}I) = f$.  Thus $\dim_{F}\im(\overline{L} - \overline{r}I) = g+1$ is a lower bound for $\dim_{F} \overline{M}_{a}$.

By Lemma \ref{lem:grassmann} with $h=2$, $a_{1} = a$, $a_{2} = a+1$, $b_{1} = g+1$, $b_{2} = g$, $b_{3} = 1$, and $d = g - 1 + ag$ we have
\begin{align*}
e_{a} &= 1,\\
e_{a+1} &= g-1,\\
e_{0} &= f,\\
e_{i} &= 0, \mbox{ otherwise.}
\end{align*}

\item[c) ]  $v_{3}(n-1) = 1$ and $v_{3}(n-4) = 1$.  The bounds $\dim_{F} \overline{M}_{1} \geq g+1$ and $\dim_{F} \overline{M}_{2} \geq g$ are obtained in the same manner as part a) of this case.  Applying Lemma \ref{lem:grassmann} with $h=2$, $a_{1}=1$, $a_{2} = 2$, $b_{1} = g+1$, $b_{2} = g$, $b_{3} = 1$, and $d = 2g-1$ we obtain
\begin{align*}
e_{1} &= 1,\\
e_{2} &= g-1,\\
e_{0} &= f,\\
e_{i} &= 0, \mbox{ otherwise.}
\end{align*}

\item[d) ]  $v_{3}(n) = a$, with $a>1$.  Then we must have $v_{3}(n-3) = 1$.  The bounds $\dim_{F}\overline{M}_{1} \geq f+1$ and $\dim_{F} \overline{M}_{a+1} \geq f$ are obtained in the same manner as part a) of this case.  Applying Lemma \ref{lem:grassmann} with $h=2$, $a_{1} = 1$, $a_{2} = a+1$, $b_{1} = f+1$, $b_{2} = f$, $b_{3} = 1$, and $d = af - a + f$ we obtain
\begin{align*}
e_{1} &= 1,\\
e_{a+1} &= f-1,\\
e_{0} &= g,\\
e_{i} &= 0, \mbox{ otherwise.}
\end{align*}

\item[e) ]  $v_{3}(n-3) = a$, with $a>1$.  Then we must have $v_{3}(n) = 1$.  Here we need $\dim_{F}\overline{M}_{a} \geq f+1$ and $\dim_{F} \overline{M}_{a+1} \geq f$.  This is shown in the same manner as in part b) of this case. Applying Lemma \ref{lem:grassmann} with $h=2$, $a_{1} = a$, $a_{2} = a+1$, $b_{1} = f+1$, $b_{2} = f$, $b_{3} = 1$, and $d = f - 1 + af$ we obtain
\begin{align*}
e_{a} &= 1,\\
e_{a+1} &= f-1,\\
e_{0} &= g,\\
e_{i} &= 0, \mbox{ otherwise.}
\end{align*}

\item[f) ]   $v_{3}(n) = 1$ and $v_{3}(n-3) = 1$.  We need $\dim_{F}\overline{M}_{1} \geq f+1$ and $\dim_{F} \overline{M}_{2} \geq f$.  This is shown in the same manner as in part a) of this case. Applying Lemma \ref{lem:grassmann} with $h=2$, $a_{1} = 1$, $a_{2} = 2$, $b_{1} = f+1$, $b_{2} = f$, $b_{3} = 1$, and $d = 2f -1$ we obtain
\begin{align*}
e_{1} &= 1,\\
e_{2} &= f-1,\\
e_{0} &= g,\\
e_{i} &= 0, \mbox{ otherwise.}
\end{align*}

\end{enumerate}

\underline{Case 3}: $p = 2$.  

\begin{enumerate}
\item[a) ]  $n \equiv 3 \pmod 4$. Then $v_{2}(n-3) = a$, with $a>1$.  We must have $v_{2}(n-1) = 1$.  We will need the bound $\dim_{F}\overline{M}_{a-1} \geq f+1$.  By Lemma \ref{lem:eigenspace} applied to the eigenvalue $r$, we see that $\dim_{F}\overline{M}_{a-1} \geq f$.  We can improve this bound to the desired one by noting that the trivial module $D^{0}$ spanned by the all-one vector lies in $\overline{M}_{a-1}$, but $D^{0}$ is not a composition factor of $S^{1}$ by Lemma \ref{lem:james}.  

Applying Lemma \ref{lem:grassmann} with $h=1$, $a_{1}=a-1$, $b_{1} = f+1$, $b_{2} = 1$, and $d = af-f$ we conclude that 
\begin{align*}
e_{a-1} &= f,\\
e_{0} &= g,\\
e_{i} &= 0, \mbox{ otherwise.}
\end{align*}

\item[b) ]  $n \equiv 2 \pmod 4$. A quick glance at equation \ref{crit-order} shows that $2$ does not divide the order of $K(\Gamma)$, so there is nothing to show.

\item[c) ]  $n \equiv 1 \pmod 4$. Then $v_{2}(n-1) = a$, with $a>1$.  We must have $v_{2}(n-3) = 1$.  We need the bound $\dim_{F}\overline{M}_{a-1} \geq g$.  This is immediate from Lemma \ref{lem:eigenspace} applied to the eigenvalue $s$.  

Applying Lemma \ref{lem:grassmann} with $h=1$, $a_{1}=a-1$, $b_{1} = g$, $b_{2} = 1$, and $d = ag-a-g+1$ we conclude that 
\begin{align*}
e_{a-1} &= g-1,\\
e_{0} &= f+1,\\
e_{i} &= 0, \mbox{ otherwise.}
\end{align*}

\item[d) ]  $n \equiv 0 \pmod 4$.  To aid with some computations to follow, it will be useful to view $F^{V(\Gamma)}$ as consisting of formal $F$-linear combinations of $2$-subsets.   That is,
\[
F^{V(\Gamma)} = \left \{ \sum a_{\{i,j\}} \{i,j\} \, \vert \, a_{\{i,j\}} \in F, \{i,j\} \subseteq \{1, 2, \ldots, n\} \right \}.
\]
In this formalism the all-one vector $\allone$ can be written $\sum_{1 \leq i < j \leq n} \{i, j\}$.
\begin{enumerate}
\item[i) ] $v_{2}(n) = a$ with $a >2$.  Then we must have $v_{2}(n-4) = 2$.   The bounds we need this time are $\dim_{F}\overline{M}_{1} \geq g+1$ and $\dim_{F}\overline{M}_{a} \geq f$.  By Lemma \ref{lem:eigenspace} applied to eigenvalue $s$ we get $\dim_{F}\overline{M}_{1} \geq g$.  Since $S^{2}$ does not have the trivial module $D^{0}$ as a composition factor by Lemma \ref{lem:james} we can sharpen this bound to $\dim_{F}\overline{M}_{1} \geq g+1$ since $\overline{M}_{1}$ contains the trivial module spanned by $\allone$.  

Now we argue that $\dim_{F}\overline{M}_{a} \geq f$.  Consider the nondegenerate symmetric $\Sym(n)$-invariant bilinear form on $F^{V(\Gamma)}$ having $V(\Gamma)$ as orthonormal basis.  Then $\allone^{\perp}$ consists of the linear combinations of $2$-subsets with coefficients summing to zero.  Of course the map defined by the all-ones matrix $J$ is the zero transformation on such a vector.  Now rewrite equation \ref{eqn:srgL} as
\begin{equation}
L(L-(r+s)I) = \mu J - rsI.
\end{equation}
Since $v_{2}(rs) = a$, it follows from this equation that 
\begin{equation} \label{eqn:perp-bound}
\overline{L}(\allone^{\perp}) \subseteq \overline{M}_{a}.
\end{equation}
  From Lemma \ref{lem:james} it follows that $F^{V(\Gamma)}$ has exactly two trivial composition factors.  Since $\allone \subseteq \Ker(\overline{L})$, the image of $\overline{L}$ can have at most one trivial composition factor.  Direct computation shows that $\overline{L}\left ( \sum_{k \in \{2, \ldots , n\}} \{1, k\} \right ) = \allone$ and that $\overline{L}( \{1, 2\})$ is not in the span of $\allone$.  Thus
\begin{align*}
\dim_{F}\im(\overline{L}) &\geq 1 + \min\left ( \dim_{F}D^{1}, \dim_{F}D^{2} \right )\\
 &= 1 + (f-1)\\ 
 &= f.
\end{align*}

Therefore $\dim_{F} \overline{L}(\allone^{\perp}) \geq f-1$ since $\allone^{\perp}$ is codimension one in $F^{V(\Gamma)}$, and so $\overline{L}(\allone^{\perp})$ must have a nontrivial composition factor.  By equation \ref{eqn:perp-bound} $\overline{M}_{a}$ has a nontrivial composition factor.  Since $\allone \subseteq \overline{M}_{a}$, we have shown
\begin{align*}
\dim_{F}\overline{M}_{a} &\geq 1 + \min\left ( \dim_{F}D^{1}, \dim_{F}D^{2} \right )\\
 &= 1 + (f-1)\\ 
 &= f.
\end{align*}

Now applying Lemma \ref{lem:grassmann} with $h=2$, $a_{1} = 1$, $a_{2} = a$, $b_{1} = g+1$, $b_{2} = f$, $b_{3} = 1$, and $d = af - a + g - f + 1$ we obtain
\begin{align*}
e_{1} &= g+1-f,\\
e_{a} &= f-1,\\
e_{0} &= f,\\
e_{i} &= 0, \mbox{ otherwise.}
\end{align*}

\item[ii) ] $v_{2}(n-4) = a$ with $a >2$.  Then we must have $v_{2}(n) = 2$.  The bounds we need in this case are $\dim_{F}\overline{M}_{a-1} \geq g+1$ and $\dim_{F} \overline{M}_{a} \geq f$.  The first bound holds by Lemma \ref{lem:eigenspace} applied to eigenvalue $s$ and noting that $[ S^{2} : D^{0}] = 0$ by Lemma \ref{lem:james}.  The second bound is obtained by an identical argument as that in part i).

Applying Lemma \ref{lem:grassmann} with $h=2$, $a_{1} = a-1$, $a_{2} = a$, $b_{1} = g+1$, $b_{2} = f$, $b_{3} = 1$, and $d = f-1+ag-g$ we obtain
\begin{align*}
e_{a-1} &= g+1-f,\\
e_{a} &= f-1,\\
e_{0} &= f,\\
e_{i} &= 0, \mbox{ otherwise.}
\end{align*}

\item[iii) ] $v_{2}(n) = 2$ and $v_{2}(n-4) = 2$.  This case actually cannot occur, since $v_{2}(n) = 2$ forces $v_{2}(n-4) \geq 3$.

\end{enumerate}

\end{enumerate}

\section{The invariant factors of $L$} \label{sec:inv}
We now combine the information in the previous section to obtain a nice description of the Smith normal form of $L$.  We will write $\Z_{m}$ to denote the group $\ZZ / m\ZZ$.

\begin{theorem}
Let $n \geq 5$.  Then the critical group of the Kneser graph $\Gamma = KG(n,2)$ has the form
\[
K(\Gamma) \cong \begin{cases}
\Z_{n-4} \oplus \left ( \Z_{\frac{(n-4)(n-1)}{2}} \right )^{\frac{n(n-5)}{2}} \oplus \Z_{\frac{(n-4)(n-1)(n-3)}{4}} \oplus \left ( \Z_{\frac{(n-4)(n-1)(n-3)n}{4}} \right )^{n-2} & \mbox{if $n$ is odd,} \\
\Z_{\frac{n-4}{2}} \oplus \left ( \Z_{\frac{(n-4)(n-1)}{2}} \right )^{\frac{n(n-5)}{2}} \oplus \Z_{\frac{(n-4)(n-1)(n-3)}{2}} \oplus \left ( \Z_{\frac{(n-4)(n-1)(n-3)n}{4}} \right )^{n-2} & \mbox{if $n$ is even.}
\end{cases}
\]
\end{theorem}
\begin{proof}
As was observed earlier, a prime that is not $2$ or $3$ can divide at most one of $n, n-1, n-3, n-4$.  We fix the notation $\pi_{m} \coloneqq \frac{m}{2^{x}3^{y}}$, where $v_{2}(m)=x$ and $v_{3}(m)=y$.  Thus $\pi_{n}, \pi_{n-1}, \pi_{n-3}, \pi_{n-4}$ are pairwise coprime.  For a prime $p$ we let $Syl_{p}$ denote the Sylow-$p$ subgroup of $K(\Gamma)$.

It follows from Case $1$ of Section \ref{sec:elem} that 
\begin{align}
K(\Gamma) &\cong \left ( \Z_{\pi_{n}}\right )^{f-1} \oplus \left ( \Z_{\pi_{n-3}}\right )^{f} \oplus \left ( \Z_{\pi_{n-1}}\right )^{g-1} \oplus \left ( \Z_{\pi_{n-4}}\right )^{g} \oplus Syl_{3} \oplus Syl_{2} \nonumber \\
&\cong \Z_{\pi_{n-4}} \oplus \left ( \Z_{\pi_{(n-4)(n-1)}}\right )^{g-1-f} \oplus \Z_{\pi_{(n-4)(n-1)(n-3)}} \oplus \left ( \Z_{\pi_{(n-4)(n-1)(n-3)n}}\right )^{f-1} \oplus Syl_{3} \oplus Syl_{2} \label{eqn:finalp}.
\end{align}

Now we use the results of Case 2 of Section \ref{sec:elem} to deal with $Syl_{3}$.  We fix the notation $\tau_{m} \coloneqq \frac{m}{2^{x}}$, where $v_{2}(m)=x$.  From Case 2, part a) we have 
\begin{align*}
e_{1} &= 1,\\
e_{a+1} &= g-1,\\
e_{0} &= f,\\
e_{i} &= 0, \mbox{ otherwise,}
\end{align*}
where $a = v_{3}(n-1)$ and $v_{3}(n-4)=1$.  Thus equation \ref{eqn:finalp} becomes
\begin{align*}
K(\Gamma) &\cong \Z_{3\pi_{n-4}} \oplus \left ( \Z_{3^{a+1}\pi_{(n-4)(n-1)}}\right )^{g-1-f} \hspace{-.3cm}\oplus \Z_{3^{a+1}\pi_{(n-4)(n-1)(n-3)}} \oplus \left ( \Z_{3^{a+1}\pi_{(n-4)(n-1)(n-3)n}}\right )^{f-1} \oplus  Syl_{2}\\
&= \Z_{\tau_{n-4}} \oplus \left ( \Z_{\tau_{(n-4)(n-1)}}\right )^{g-1-f} \oplus \Z_{\tau_{(n-4)(n-1)(n-3)}} \oplus \left ( \Z_{\tau_{(n-4)(n-1)(n-3)n}}\right )^{f-1} \oplus  Syl_{2}.
\end{align*}
In the same way one can move through the remaining parts of Case 2 to see how $Syl_{3}$ absorbs into the prior cyclic factors. Interestingly, all cases work out to the same cyclic decomposition:
\begin{equation} \label{eqn:final3}
K(\Gamma) \cong \Z_{\tau_{n-4}} \oplus \left ( \Z_{\tau_{(n-4)(n-1)}}\right )^{g-1-f} \oplus  \Z_{\tau_{(n-4)(n-1)(n-3)}} \oplus \left ( \Z_{\tau_{(n-4)(n-1)(n-3)n}}\right )^{f-1} \oplus  Syl_{2}.
\end{equation}

Finally, we look at how equation \ref{eqn:final3} appears when considering the possibilities for $Syl_{2}$ in Case 3 of Section \ref{sec:elem}.  If $n \equiv 3 \pmod{4}$, then 
\begin{align*}
K(\Gamma) &\cong \Z_{\tau_{n-4}} \oplus \left ( \Z_{\tau_{(n-4)(n-1)}} \right )^{g-1-f} \oplus \Z_{2^{a-1}\tau_{(n-4)(n-1)(n-3)}} \oplus \left ( \Z_{2^{a-1}\tau_{(n-4)(n-1)(n-3)n}}\right )^{f-1}\\
&= \Z_{n-4} \oplus \left ( \Z_{\frac{(n-4)(n-1)}{2}} \right )^{g-1-f} \oplus \Z_{\frac{(n-4)(n-1)(n-3)}{4}} \oplus \left ( \Z_{\frac{(n-4)(n-1)(n-3)n}{4}} \right )^{f-1},
\end{align*}
since $v_{2}(n-3) = a$ and $v_{2}(n-1) = 1$.  If $n \equiv 1 \pmod{4}$, then
\begin{align*}
K(\Gamma) &\cong \Z_{\tau_{n-4}} \oplus \left ( \Z_{2^{a-1}\tau_{(n-4)(n-1)}} \right )^{g-1-f} \oplus \Z_{2^{a-1}\tau_{(n-4)(n-1)(n-3)}} \oplus \left ( \Z_{2^{a-1}\tau_{(n-4)(n-1)(n-3)n}}\right )^{f-1}\\
&= \Z_{n-4} \oplus \left ( \Z_{\frac{(n-4)(n-1)}{2}} \right )^{g-1-f} \oplus \Z_{\frac{(n-4)(n-1)(n-3)}{4}} \oplus \left ( \Z_{\frac{(n-4)(n-1)(n-3)n}{4}} \right )^{f-1},
\end{align*}
since $v_{2}(n-1) = a$ and $v_{2}(n-3) = 1$.  If $n \equiv 0 \pmod{4}$ with $v_{2}(n) = a (a>2)$ and $v_{2}(n-4) = 2$, we have
\begin{align*}
K(\Gamma) &\cong \Z_{2\tau_{n-4}} \oplus \left ( \Z_{2\tau_{(n-4)(n-1)}} \right )^{g-1-f} \oplus \Z_{2\tau_{(n-4)(n-1)(n-3)}} \oplus \left ( \Z_{2^{a}\tau_{(n-4)(n-1)(n-3)n}}\right )^{f-1}\\
&= \Z_{\frac{n-4}{2}} \oplus \left ( \Z_{\frac{(n-4)(n-1)}{2}} \right )^{g-1-f} \oplus \Z_{\frac{(n-4)(n-1)(n-3)}{2}} \oplus \left ( \Z_{\frac{(n-4)(n-1)(n-3)n}{4}} \right )^{f-1}.
\end{align*}
If $n \equiv 0 \pmod{4}$ with $v_{2}(n-4) = a (a>2)$ and $v_{2}(n) = 2$, we have
\begin{align*}
K(\Gamma) &\cong \Z_{2^{a-1}\tau_{n-4}} \oplus \left ( \Z_{2^{a-1}\tau_{(n-4)(n-1)}} \right )^{g-1-f} \oplus \Z_{2^{a-1}\tau_{(n-4)(n-1)(n-3)}} \oplus \left ( \Z_{2^{a}\tau_{(n-4)(n-1)(n-3)n}}\right )^{f-1}\\
&= \Z_{\frac{n-4}{2}} \oplus \left ( \Z_{\frac{(n-4)(n-1)}{2}} \right )^{g-1-f} \oplus \Z_{\frac{(n-4)(n-1)(n-3)}{2}} \oplus \left ( \Z_{\frac{(n-4)(n-1)(n-3)n}{4}} \right )^{f-1}.
\end{align*}
Thus we see that $K(\Gamma)$ depends only on the parity of $n$ and the theorem is proved.

\end{proof}
\bibliographystyle{amsplain}
\bibliography{2-setsbib}

\end{document}